\documentclass[12pt]{amsart}

 \usepackage{
 amsfonts,
 latexsym,
 amssymb,
 amsbsy,
 enumerate,
 mathrsfs,
 nicefrac}
\usepackage[a4paper, left=1in,right=1in,top=1in,bottom=1in]{geometry}

 \newcommand{\labbel}{\label}

 \newtheorem{theorem}{Theorem}[section]

 \newtheorem{prop}[theorem]{Proposition}

 \newtheorem{corollary}[theorem]{Corollary}

 \newtheorem*{theorem*}{Theorem}
 \newtheorem*{corollary*}{Corollary}
 \newtheorem*{proposition*}{Proposition}

 \theoremstyle{definition}
 \newtheorem{definition}[theorem]{Definition}

 \newtheorem*{definition*}{Definition}

 \theoremstyle{remark}
 \newtheorem{remark}[theorem]{Remark}

 \newtheorem*{disclaimer*}{Disclaimer}
 \newtheorem*{acknowledgement}{Acknowledgement}

 \newcommand{\brfrt}{\hspace{0 pt}}

 \DeclareMathOperator{\cf}{cf}
 \DeclareMathOperator{\slim}{slim}
 \DeclareMathOperator{\sliminf}{slim\,inf}
 \DeclareMathOperator{\slimsup}{slim\,sup}
 \DeclareMathOperator{\Lim}{Lim}
 \DeclareMathOperator{\dom}{dom}
 \DeclareMathOperator{\sh}{sh}

\begin{document}
 \title{Surreal limits}

 \author{Paolo Lipparini}
 \address{Dipartimento di Matematica\\Viale della Ricerca Scientifica\\
 Universit\`a di Roma ``Tor Virgolettata'' \\I-00133 ROME ITALY}
 \urladdr{http://www.mat.uniroma2.it/\textasciitilde lipparin}
 \author{Istv\'an Mez\H{o}}
 \address{Department of Mathematics\\Nanjing University of Information Science and Technology\\
Nanjing, 210044, P. R. China}
 \urladdr{https://sites.google.com/site/istvanmezo81/}



 \thanks{Paolo  Lipparini has been  partially supported by 
PRIN 2012 ``Logica, Modelli e Insiemi'', and
his work has been performed under the auspices of G.N.S.A.G.A.  \\The research of Istv\'an Mez\H{o} was supported by the Scientific Research Foundation of Nanjing University of Information Science \& Technology, the Startup Foundation for Introducing Talent of NUIST. Project no.: S8113062001, and the National Natural Science Foundation for China. Grant no. 11501299.}

 \begin{abstract}
 We note that if a sequence of real
 numbers converges to some limit, then
 the sequence of the corresponding strings  in the   surreal
 $+,-$ sign expansion
 representation converges, for  a natural notion
 of string convergence,
 to the string corresponding to the  limit,
 modulo an infinitesimal.
 The corresponding statement  would be obviously false if we
 were considering, as strings,  decimal or binary representations,
 instead. The string limit
 of a possibly transfinite
 sequence of surreal numbers is always defined and, when considering
  increasing sequences of ordinals, corresponds to
 taking the supremum.
 A transfinite sum can be defined
 using the string limit
 and this sum agrees with the
 representation of a surreal number in
 Conway normal form.
 \end{abstract}

 \maketitle

 \section{Introduction} \labbel{intro}

 \subsection*{Limits of sequences from the surreal point of view.} If we compute, say, the number $e$ by means of the usual series expansion
 $e=\sum _{n=0}^{ \omega } \frac{1}{n!}$,
 we get the following sequence of partial sums:
 \[
 1, \ \, 2, \ \, 2.5, \ \, 2.666\dots, \ \, 2.7083\dots, \ \,  2.7166\dots, \ \,  2.7180\dots, \ \,  2.7182\dots
  \]

 Since the digits in the above decimal expressions eventually stabilize, we get
 confident that $2.718\dots$  approximates the actual decimal expansion of $e$
 (of course, this is not a proof!).
 Though a similar argument goes well for many limits,
 it is not always the case that digits eventually stabilize
  in converging sequences. For example, consider the sequence
 \[
 1.99, \ \ 2.01, \ \ 1.999, \ \  2.001, \ \ 1.9999, \ \  2.0001, \ \ \dots,
  \]
  which obviously converges to $2$. In this case, no digit at all stabilizes.

 We observe that, instead, if
 we represent real numbers
   as sequences
 of $+$'s and $-$'s by means of  the surreal
  sign expansion, and we consider a converging (in the classical sense
 of real analysis) sequence
 of real numbers, then the $+$'s and $-$'s which stabilize
 give the surreal sign expansion of the limit,
 possibly with the difference of  an infinitesimal. See Theorem \ref{thm}.
 This shows that the apparently odd and strange surreal
 way of representing numbers, in particular, reals, by
  means of a sequence of $+$'s and  $-$'s is, in a sense, actually
 a more natural way,  in comparison with  the usual decimal or binary expansions.
 So far, the above result is just slightly more
 than a mere curiosity,
 but a few arguments presented below suggest at least some
 remote eventuality of further  developments.

\subsection*{Some history} 
 We refer to Conway \cite{C},  Ehrlich \cite{E}, Gonshor \cite{G} and Siegel \cite{S}
  for details and history
 about surreal numbers, in particular, for details about
 the sign expansion.

 The usual $\varepsilon$-$\delta$  definition
 cannot be applied to sequences of surreals, at least
 if $\varepsilon$ is intended to vary among all surreals,
 since there are always
 plenty of surreals between the elements of a sequence and any
 purported limit. Conway himself wrote about this situation as follows \cite{C}:

''For instance, the limit of the sequence $0,\frac{1}{2},\frac{2}{3},\frac{3}{4},\dots$ ($\omega$ terms) is not $1$, at least in the ordinary sense, because there are plenty of numbers in between. A simpler, but sometimes less convincing, example of the same phenomenon is given by the sequence $0,1,2,3,\dots$ of all finite ordinals, which one would expect to tend to $\omega$, but which obviously can't since there is a whole Host of numbers greater than every finite integer but less than $\omega$. For the author's amusement, we recall some of the simplest of them:
\[\omega-1,\;\omega/2,\;\sqrt{\omega},\;\omega^{1/\omega},\dots\mbox{''}\]

 \subsection*{Definitions of limit up to now.} There are some existing approaches to overcome with the difficulties described above. 
Here we unify and extend the approaches from \cite{M,L2},
which appeared as archived but not otherwise published
manuscripts.

A surreal notion of convergence has been introduced by  Mez\H{o} \cite{M}.  The idea from \cite{M} is to consider, as possible limits,  only
 surreals of length at most the superior limit of the lengths of the elements
 of the sequence. Then convergence is established in the way hinted
 above, considering sign persistence. In fact,  the  definition from \cite{M},
 when taken literally, imports that if $s$ is a limit of a sequence,
 then every initial segment of $s$ is a limit, too.
  In this sense, for countable
 sequences,
 we shall consider here  the unique longest  Mez\H{o}'s limit.
 In \cite{M}   countably infinite sums are also considered.

 In another direction, Rubinstein-Salzedo and Swaminathan \cite{RS}
  noticed that the   $\varepsilon$-$\delta$ condition
 makes sense for the surreals, too, provided one considers class $On$-long
 sequences. For such class-sequences, the definition given in \cite[Definition 19]{RS}
  seems to bear a deep
 resemblance with the present one 
(see the beginning of Section \ref{fur}), though the relationships
 between the two notions have  not been fully analyzed yet.

 Then Lipparini \cite{L2}
  rediscovered the surreal limit and generalized it to
 (set) ordinal-indexed sequences.
 His work was motivated by
 the study of  transfinite iterations
 of the natural sum on the ordinals. See \cite{L}.
 Indeed, in the sense of surreal sign expansions,
  an ordinal
 is a sequence of $+$'s. If we have an increasing sequence of ordinals,
 then every place in the sequence is stable from some point on, hence
 a construction similar to the above one furnishes the limit, that is, the supremum,
 of the sequence of ordinals.
 See Section \ref{su} for more details and for the transfinite sum operation arising from
 the construction.

 The above idea of limit  can be carried over in general,
 dealing with sequences of
 transfinite strings of symbols, or, even more generally,
 sequences of labeled linearly ordered sets. In this sense,
 Lipparini \cite{L2} has introduced  a rather general notion; see
 Section \ref{fur} here, in particular,
   Definition \ref{def}. However, in the next section
  we shall be content
 with the particular case of countable sequences of
 surreal numbers, mostly, reals.
 Then in Section \ref{su} we shall deal with
  transfinite sequences of surreals.

 The present version  of the paper has been mostly written by  Paolo Lipparini.
 He is the sole responsible for any error, omission or inaccuracy.

 \section{Surreal limits of real numbers} \labbel{real}

\subsection*{The $s$-limit.} For notational convenience,
 we shall identify a surreal number with its
 sign expansion, that is, its representation as
 an ordinal-indexed  sequence of $+$'s and  $-$'s.
 Thus a surreal number is a function $s: \alpha_s \to \{ +, - \}  $,
 where $\alpha_s$ is an ordinal depending on $s$.
 The ordinal
 $\alpha_s$ will be called the \emph{length} of $s$;
 in the original theory developed by J. Conway,
 $\alpha_s$ corresponds to  the \emph{birthday} of $s$.
 If $\beta< \alpha _s$,  $s( \beta )$ will be sometimes called
 \emph{place $\beta$ (in the representation) of $s$}, and places
 $ \geq \alpha_s$ will be dubbed \emph{undefined}.

 \begin{definition} \labbel{def1}
 If $( s_n ) _{ n < \omega } $  is a sequence of surreal numbers,
 we define the \emph{s-limit
   of
 $( s_n ) _{ n < \omega } $},
 in symbols, $\slim_{n < \omega} s_n  $
 as the surreal $s$ such that place $\gamma$
  in the sign expansion of $s$ is defined if and only if there is $m< \omega $
 such that, for every $n \geq m$, the sign expansions of
 the $s_n$'s are identical (and defined) up to place $\gamma$ included.
 If this is the case, place  $\gamma$
 of $s$ is set to be equal to the corresponding place of
 $s_m$ (hence also of the  $s_n$'s which follow).
 Notice that, by construction,
 if $s(\gamma)$ is defined, then
 $s(\gamma')$ is defined, too, for every $\gamma' < \gamma $,
 hence the definition is well posed.
  \end{definition}

 The s-limit of a sequence  is always defined;
 possibly, it is the empty sequence.
 Notice also that, as far as we meet a place
 at which the values of the $s_n$'s eventually oscillate,
 we impose that the sign expansion of the s-limit
  stops at that point. This is necessary
 if we want Theorem \ref{thm} below to become true.. 
 As an additional  reason in support of the choice,
 we shall note at the beginning of
Section \ref{fur} 
  that the s-limit admits a natural game-theoretical
 definition in terms  of Left and Right options.

\subsection*{The connection between the classical limit and $s$-limit.} Since a real number is
 (or can be considered as) a surreal number,
 we have also the notion of an s-limit of a sequence of reals.
 The main observation in the present section
 is the curious fact that
 if a sequence of real numbers has a limit in the sense of classical
 real analysis, then the limit and the s-limit coincide,
 modulo an infinitesimal.

 In the proof of the following theorem
 we shall freely use the Berle\-kamp's Sign-Expansion Rule
 and the characterization of surreals born at day $ \omega$. See, e.~g.,
 \cite[VIII, 2]{S} for full details. Recall that, under
 the sign expansion representation, the
 surreal order is obtained by comparing
 the first difference. This is quite similar to
 the lexicographic order, but formally different;
 in the surreal sense undefined is considered to be between
 $-$ and  $+$, rather than before them.
 Notice that the following theorem
 deals only with surreals born at most at day $ \omega$.
 We shall always consider \emph{addition} $+$ in the surreal sense
 (symbols overlapping never causes confusion).
 Of course, when restricted to real numbers, surreal addition
 coincides with usual addition.

 \begin{theorem} \labbel{thm}
 If $( r_n) _{n < \omega } $
 is a sequence of real numbers
 and
 $\lim_{n\to\infty}  r_n $
 exists, possibly equal to $ \infty= \omega$
  or $ -\infty = -\omega$,
 then
 $\lim_{n\to\infty}  r_n  = \slim_{n < \omega } r_n + \varepsilon $,
 where $\varepsilon$ is either
 $0$, or
 $ 1/ \omega $, or
 $-1/ \omega $.
   \end{theorem}

  \begin{proof}
 Let $r=\lim_{n\to\infty}  r_n$. First, suppose that  $r=0$.
 In the  sense of sign expansions, this corresponds to the
 empty sequence.
 If   infinitely many
 $r_n$'s are equal to $0$,
 then
 $\slim_{n < \omega } r_n =0$;
 indeed, in this case, the first sign
 cannot eventually stabilize \emph{and be defined},
 hence the definition of the s-limit gives the empty sequence.
 If all but finitely many
 $r_n$'s are
 strictly positive,
 then the first sign in the surreal representation
 is eventually $+$.
 Since the sequence converges in the classical sense,
 then, for every $m < \omega$,
 we have $r_n \leq 1/2^m$  eventually, that is,
 the first $+$ sign is eventually followed by at least $m$ minuses.
 This means that  $\slim_{n < \omega } r_n = +---\dots = 1/ \omega $.
 If all but finitely many
 $r_n$'s are
 strictly negative,
 the symmetric argument gives
 $\slim_{n < \omega } r_n = -1/ \omega $.
 In the remaining case we have
 infinitely many positive $r_n$'s, hence
 infinitely many occurrences of $+$ at the first place, and
  infinitely many negative $r_n$'s, hence
 infinitely many occurrences of $-$ at the first place.
 Thus the first place does not eventually stabilize and,
 according to
 Definition \ref{def1},
 $\slim_{n < \omega } r_n =0$,
 being the empty sequence.

 The case when $r$ is a dyadic rational is similar.
 A number $r'$  greater than $r$ and sufficiently close to $r$
 has the form $r+---\dots$, where juxtaposition denotes string concatenation
 and, as far as $r'$ approaches $r$, after the string $r+$
 we get a larger and larger number
 of minuses,   possibly followed by a plus and other
 signs.
 Thus if $( r_n) _{n < \omega } $ converges to
 $r$ eventually from above,
 $\slim_{n < \omega } r_n = r+---\dots = r+  \nicefrac{1}{ \omega}  $.
 Symmetrically,
 if $( r_n) _{n < \omega } $ converges to
 $r$ eventually from below,
 $\slim_{n < \omega } r_n  = r-  \nicefrac{1}{ \omega} $.
 In the remaining cases, either  $r_n=r$ for infinitely many $n$'s,
 or there are infinitely many  $n$'s such that
 $r_n<r$ and there are infinitely many  $n$'s such that
 $r_n>r$. In each of the above cases
 $\slim_{n < \omega } r_n =r$.

 If $r$ is real and not dyadic,
 then its sign expansion is infinite and
  neither eventually $+$ nor eventually $-$.
 Suppose that $r>0$, the case $r<0$ being treated symmetrically.
 Since $r$ is not dyadic, then, in particular, it is not an integer.
 Letting $[r]$ denote the integer part of $r$,
 we have that $|r-r_n| < \min(r-[r], [r]+1-r)$,
 for sufficiently large $n$,
 hence, from some point on, the integer part of  $r_n$
  is the same as the integer part of $r$; moreover, $r_n$
 ha a binary point, too, hence these
 parts of the sign expansion eventually stabilize.
 What remain are  the fractional parts,
 which are computed like the binary expansion, except possibly for
 a last sign/digit. Since the sign
 expansion of $r$ is neither eventually $+$ nor eventually $-$,
 then, for every $ m < \omega$, both a $+$ and a $-$ occur
 after the $m^{\rm th}$ place of the fractional part of $r$.
 Let $q > m$ be such that both  a $+$ and a $-$ occur
 between the  $m+1^{\rm th}$ place and
 the $q^{\rm th}$ place of the fractional part of $r$. If
  $|r-r_n| < 2^q$, then
 $r$ and $r_n$ have the same fractional part up to the
 $m^{\rm th}$ place.
 Since $m$ is arbitrary,
 the fractional part of the sequence
 $( r_n) _{n < \omega } $  eventually stabilizes to the
 value of the fractional part of $r$.
 In conclusion, the whole sign expansions stabilize
 to the sign expansion of $r$.

 Notice that if $r$ is not  a dyadic rational, the same argument works
 for the binary expansions, as well. Hence,
 in the sense of the theorem,
 the sign expansion representation has
  some
 actual advantage
 over the binary one
  only for the countable set
 of dyadic rationals.

 Finally, the cases when $r=\infty$ or
 $r=-\infty$ are trivial.
  \end{proof}

 The eventual presence  of an infinitesimal in
 Theorem \ref{thm} does not seem to
 be a serious drawback.
 Anyone  interested only in
 real numbers would surely feel free to ignore it.
 In a sense, the presence of  $\pm 1/ \omega $ has some use,
 since it tells us when the sequence converges from above or from below.
 However, notice the asymmetry between the cases of dyadic  and nondyadic
 limits, since the infinitesimal can appear only in the former case.

 In the above respect, a more uniform version of
 Theorem \ref{thm} holds, with exactly the same proof.
 A surreal $s$ is \emph{finite}  if $ -n < s < n $,
for some natural number $n$.
 Any finite surreal can be expressed uniquely as a sum
 $s= R(s) + \varepsilon(s) $,
 where
 $ R(s)\in \mathbb R $ and $  \varepsilon(s)$ is infinitesimal.
 Then the proof of Theorem \ref{thm} together
 with  Berlekamp's  Rule and easy facts about
 the sign expansion shows the following.

 \begin{corollary} \labbel{realx}
 If $( s_n) _{n < \omega } $
 is a sequence of finite surreal numbers
 and
 $\lim_{n\to\infty}  R(s_n) $
 exists and is finite,
 then
 $\lim_{n\to\infty}  R(s_n)  = R(\slim_{n < \omega } s_n )$.
  \end{corollary}

\subsection*{Comparison of limits} \labbel{comp}
 The notion of limit in the classical sense
 and the notion of s-limit do not coincide, in general,
 even modulo infinitesimals.
 Indeed, the latter limit is always defined, while this is not
 necessarily the case for the former.
 As we mentioned during the proof, Theorem \ref{thm}
 shows that  the sign expansion has some real advantage
  over the binary expansion only in the case of dyadic numbers.
 However, independently from Theorem \ref{thm}, the sign expansion has the advantage of  representing
 each real \emph{uniquely},
 while the choice of, say, the binary representation
 $1.000\dots$ in place of $0.111\dots$ for the natural number $1$
 might be perceived as somewhat
 arbitrary.

 We say that the s-limit $s$  of some sequence $( s_n ) _{ n < \omega } $
   is a \emph{full s-limit}
 if the length of $s$ is the inferior limit of the lengths of the $s_n$'s;
 in other words, if $s$ is as long as possible,
 as far as this is compatible with the lengths of
 the $s_n$'s. The relationship between the classical limit and
 the s-limit reverses, if we take into account only full limits.
 Indeed, if
 $( r_n) _{n < \omega } $ is a sequence of real numbers and
 $s$ is a full s-limit of  $( r_n) _{n < \omega }$,
 then
 $\lim_{n\to\infty}  r_n  $
 exists (possibly $ \omega $ or $- \omega $)
 and is equal to $s+ \varepsilon $,
 for $\varepsilon$  either
 $0$, or
 $ 1/ \omega $, or
 $-1/ \omega $.
 As apparent from the proof
 of Theorem \ref{thm}, not every
 s-limit is full, even in case the sequence is convergent in the classical sense.
 E. g., the s-limit of the sequence $(\frac {-1}{2})^n$  is not full; indeed,
 the elements of the sequence
 are represented as $+$, $-+$, $+--$, $-+++$, \dots\
  Every monotone sequence of real numbers has a full s-limit,
 but also nonmonotone sequences can have a full s-limit,
 e.~g., $+$, $+-$, $+-+$, $+-+-$, \dots, converging to $\nicefrac{2}{3}$.

\subsection*{Inferior and superior limits} 
 The main idea behind the definition of the s-limit
 is to take into account not only full limits (by the above remark,
 this would not be
 sufficient to represent all classical real limits),
 but to treat  also the case in which signs oscillate
 at a certain place.
 The choice made in Definition \ref{def1}  is, in a sense,
 a neutral one:
 if the sign at a certain place
 is not eventually constant,
 we set the sign to be undefined in the limit.
 On the other hand, we could have chosen either the
  smallest possible value, or, in the other direction, the largest value.
 The above considerations suggest the definitions of
 $\sliminf$ and  $\slimsup$ that we shall formally give 
in the general case in 
Section \ref{fur}. 
For countable sequences of real numbers,
 in order to compute $\sliminf$, and starting from
 the first place, we choose $+$ if this is the eventual
 sign of the members of the sequence;
 we choose $-$ if there are infinitely many $-$'s
  in that place in the sequence (notice the asymmetry
 between the two cases!), and
 in this latter case
  we discard all the members of the sequence with a different sign.
 In the remaining  case (infinitely  many undefined)
  we set the place undefined in $\sliminf$
 and we stop the construction. In case the construction
 does not stop at a certain place, we proceed in the same way
with the next place, and so on.
Notice that, when dealing with real numbers or, more generally, 
surreals of length at most $ \omega$, we never go beyond
the $ \omega^{\rm th}$-step of the construction. The delicate issue 
of subsequent steps   shall be dealt
with in 
Section \ref{fur}. 

A definition symmetric to $\sliminf$ gives $\slimsup$.

 Notice that $\slim$, $\sliminf$ and  $\slimsup$
 do not necessarily strictly agree even for converging sequences of real
 numbers. Considering, as above,
 $r_n = (\frac{-1}{2}) ^n$,
 we have
 $\slim_{n < \omega } r_n = 0$,
  $\sliminf_{n < \omega } r_n = \nicefrac{-1}{ \omega } $, and
  $\slimsup_{n < \omega } r_n = \nicefrac{1}{ \omega } $.
 However, for a converging sequence of reals,
 $\slim$, $\sliminf$ and  $\slimsup$ and the classical
 real limit $\lim$ do agree modulo an infinitesimal.
 For a general, possibly non converging, sequence of reals,
 $\sliminf$  agrees with the classical $\liminf$,
 modulo an infinitesimal, and the same holds for
 $\slimsup$ and $\limsup$.

 Notice that in the definition of $\sliminf$ 
we do need discard certain members of the sequence
 in the case of infinitely many minus signs, when the sign
is  not eventually constant.
Otherwise, we would have the unwanted result that,
 for the above sequence,
  $\sliminf_{n < \omega } r_n = - \omega  $.

\subsection*{Comparison with the Limit} 
 Another notion of limit
 for $ \omega$-\brfrt indexed sequences of surreal numbers
 is known and useful. It is usually called the \emph{Limit}, denoted
by $\Lim$ 
 with upper-case L, and is obtained by taking componentwise the
 limits (in the sense of real analysis)
  of the coefficients in the Conway normal
 representation of the elements of the sequence.
 This Limit is not always defined:
  the componentwise real limits should
 always exist and be finite, and the result should
 give an actual surreal number
 (if the result contains an ascending
 sequence of exponents of $ \omega$,
 it is not a surreal number).

 Though Theorem \ref{thm} shows that
 the Limit and the s-limit give quite close
 results when taking the limit of
 a sequence of reals, on the other hand,
 for arbitrary surreals, the two limits could turn out to be
 quite remote.
 For example, $\Lim_{n\to\infty}   \frac{ \omega}{n} = 0 $,
 while $\slim_{n < \omega } \frac{ \omega}{n} = \sqrt{ \omega } $.
 This last s-limit
 might appear unnatural, but notice that
 $\sqrt{ \omega }$ is ``multiplicatively  halfway''
 between $1=  \frac{ \omega}{ \omega }$
 and $ \omega$.
 However, $\Lim_{n\to\infty}   (\frac{ \omega}{n}+1) = 1 $,
 while $\slim_{n < \omega } (\frac{ \omega}{n} +1) $ is still
 $  \sqrt{ \omega } $. This shows that
 the s-limit of a finite sum is not always
 the sum of the s-limits.
 In another direction,
 $\Lim_{n\to\infty}    \omega^n = 0 $,
 while
  $\slim_{n < \omega }   \omega^n = \omega ^ \omega  $,
 which seems much closer to intuition and corresponds
 to a general rule we shall describe
 in the next section for taking s-limits of ordinals.
 As a final example,
 $\Lim_{n\to\infty}   ( \omega - n ) $ is undefined,
 and
  $\slim_{n < \omega }   (\omega - n) =   \frac{ \omega}{ 2} $,
 again, halfway between $0= \omega - \omega $ and $ \omega$.

 Probably, there is not a unique notion of limit for a sequence of surreals
 which is good
 for every purpose; in each particular case one should choose the
 most appropriate
 notion.

 \section{Subsequences and sums} \labbel{su}

\subsection*{Subsequences.} Taking s-limits of surreal numbers is not always a monotone operation.
 For example, if  $a_ n= n $ and $b_n = \omega -1$, for every $n < \omega$,
 we have $a_n < b_n$, for every $n< \omega $,
 but  $ \slim _{n< \omega } a_n = \omega > \omega -1 =
 \slim _{n< \omega } b_n  $.
 Of course, a similar situation
 should occur for every notion of limit defined on
 \emph{all} countable sequences of surreals,
 just assuming that the limit of a constant sequence gives its
 constant value.
 Then if the limit of some sequence is greater than all
 the elements of the sequence, you can also
 find a surreal in between the limit
 and all the members of the sequence, and the same argument as above applies.

 It is not always the case that the s-limit
 of a subsequence coincides with the s-limit of the sequence.
 If $a_n= n$, for $n$ even, and
 $a_n = \omega -1$, for $n $ odd, then
 $ \slim _{n< \omega } a_n = \omega $,
 but $ \slim _{n< \omega, \text{ $n$ odd}} a_n = \omega -1 $.
 However, the s-limit of a nondecreasing sequence
 gives a result $\geq$  than all the elements of the sequence; moreover,
 a cofinal subsequence of a nondecreasing sequence has the same s-limit.

 Since the proofs of the above facts work as well
 for sequences indexed by ordinals $> \omega$,
 we shall present the general results.
 Notice that Definition \ref{def1}
 can be naturally extended to deal
 with arbitrary ordinal-indexed
 sequences; just consider those initial
 segments of sign expansions which are eventually constant.
 The reader might assume that we are always dealing with
 sequences indexed by some  limit ordinal; otherwise,
 for a sequence indexed by a successor ordinal,
 simply set the s-limit to be the last element of the sequence.
 See full details
 in Definition \ref{def} below.

 \begin{prop} \labbel{prop}
 Suppose that  $( s _ \beta  ) _{ \beta  < \alpha  } $
 is a
 nondecreasing sequence
 of surreals with s-limit  $s$. Then
   \begin{enumerate}[(a)]
     \item
  $ s_ \delta  \leq s$,
 for every $ \delta   <  \alpha $.
   \end{enumerate}

 If $( t _ \zeta  ) _{ \zeta  < \eta } $  is another
 nondecreasing sequence, with s-limit $t$, then
   \begin{enumerate}[(a)]
 \item[(b)]
 If $( t _ \zeta  ) _{ \zeta  < \eta } $ is
 (an order-preserving rearrangement of) a cofinal subsequence of
  $( s _ \beta  ) _{ \beta  < \alpha  } $, then $t=s$.
 \item[(c)]
 Suppose that  $( s _ \beta  ) _{ \beta  < \alpha  } $
 and  $( t _ \zeta  ) _{ \zeta  < \eta } $ are \emph{chained},
 in the sense that,
 for every $ \beta < \alpha $, there is
 $ \zeta  < \eta$ such that
  $s_ \beta \leq t_ \zeta $
 and conversely.
 Then $s=t$.
    \end{enumerate}
 \end{prop}

 \begin{proof}
 Let us say that two surreals
 $\gamma'$-agree if  their sign expansions
  agree up to place $\gamma'$ included
(allowing the  possibility of corresponding places to be
both undefined).

 (a) Let $ \delta   <  \alpha $,
 we want to show that
  $ s_ \delta  \leq s$.
 If, for every ordinal $\gamma$,
 place $\gamma$ in the subsequence
  $( s _ \beta  ) _{ \delta \leq \beta  < \alpha  } $ is constant
(allowing undefined values),
  then the subsequence itself is constant, the definition
 of the s-limit gives
 $ s_ \delta  = s$ and we are done.
 Otherwise, let $\gamma$ be the first ordinal
 such that place $\gamma$ in the subsequence
  $( s _ \beta  ) _{  \delta \leq \beta  < \alpha  } $ is not  constant.
 By the definition of $\gamma$, all
 the $s_\beta$'s $\gamma'$-agree, for every $\gamma' < \gamma $
 and $\beta \geq \delta $,
 hence,
 again by the definition of the s-limit, $s$ and $ s_ \delta $
 $\gamma'$-agree, for every $\gamma' < \gamma $.
 Since the $s_\beta$'s $\gamma'$-agree, for every $\gamma' < \gamma $
 and $\beta \geq \delta $,
 and the sequence is nondecreasing, the only possible
 transitions at place $\gamma$ are from $-$ to undefined or
 $+$ and from undefined to $+$. By the definition of $\gamma$,
 at least one transition occurs and, since there is a finite number
 of possible transitions
 and no cycle is possible, place $\gamma$ eventually stabilizes
 to some value, which will be the value  in $s$. Thus  $ s_ \delta < s$.

 (b) Suppose by contradiction that
 $t\not =s$
 and
 let $\gamma$ be
 the first place at which they disagree.
 Hence at least one between
 $t( \gamma )$ and
 $s ( \gamma )$ is defined.
 Suppose that $t( \gamma )$ is defined;
 the other case is similar and easier.
 By the definition of the s-limit, there is some
 $\bar {\zeta}$ such that
   $t _ \zeta   $'s $\gamma$-agree with $t$,
 for every $\zeta \geq \bar {\zeta}$.
 If
 $t _{\bar {\zeta}} \leq u \leq t _ \zeta  $
 and $t _{\bar {\zeta}} $ and $  t _ \zeta  $
 $\gamma$-agree, then they $\gamma$-agree with $u$.
 Since
 $( t _ \zeta  ) _{ \zeta  < \eta } $ is
 cofinal in
  $( s _ \beta  ) _{ \beta  < \alpha  } $
 and
  $( s _ \beta  ) _{ \beta  < \alpha  } $ is nondecreasing,
 then the $s_ \beta $'s eventually $\gamma$-agree with
 $t _{\bar {\zeta}}$, hence with $t$. Then the definition of s-limit
 gives that $s$  $\gamma$-agrees with
 $t$, a contradiction.

 (c) Consider a new sequence made by all the elements
 of both sequences, ordered in such a way that the new big
 sequence is still nondecreasing.
 Since the original sequences are nondecreasing,
 this can be accomplished in such a way that they  actually
 become subsequences of the big sequence.
 Since the original sequences are
 chained,
  they are both cofinal in the big sequence
 (except perhaps for the trivial case in which all
  the sequences are eventually constant).
 If $u$ is the  s-limit of the big sequence, then, by (b),
 $u=s$ and  $u=t$, thus $s=t$.
 \end{proof}

\subsection*{s-limits of ordinals} \labbel{ord} 
 As we hinted to in the introduction,
 the s-limit of a non decreasing ordinal-indexed sequence
 of ordinals is their supremum. In general,
 the s-limit of an ordinal-indexed sequence of ordinals is
 their inferior limit, i. e.,
 the supremum of the set of  those ordinals $\alpha$ such that
 the  members of the sequence
  are eventually $\geq \alpha$.

\subsection*{Surreal series}   \labbel{ssum}
 Since an addition operation is defined among
  surreal numbers, any notion of limit
 entails the definition of a series.
 If $( s _n ) _{ n < \omega } $  is a sequence of surreals,
  let $\sum^s _{n < \omega } s_n$ be
  $\slim_{n < \omega } S_n$,
 where  $S_n$ denotes the partial sum
 $s_0 + s_1 + \dots + s _{n-1} $.
 By Proposition \ref{prop}(c),
 if all the $s_n$'s are nonnegative, then
 $\sum^s _{n < \omega } s_n$ is invariant under permutations,
 since the corresponding partial sums are chained in the sense of
 \ref{prop}(c).
 Here we are using the general commutative-associative property
 of $+$ and the fact that the nonnegative surreals
 form an ordered monoid.

 Since
   we can define the s-limit
 of every ordinal-indexed sequence of surreals, the above
  s-sum of length $ \omega$ can be extended to the transfinite.
 Details go as follows.
 The s-sum of the empty sequence is $0$.
 If $\sum^s _{ \alpha < \beta  } s_ \alpha $
 has been already constructed, let
  $\sum^s _{ \alpha < \beta +1  } s_ \alpha
 = s _{ \beta} + \sum^s _{ \alpha < \beta  } s_ \alpha $.
 Finally, if
 $\beta$ is limit, let
  $\sum^s _{ \alpha < \beta  } s_ \alpha
 = \slim _{ \beta' < \beta }  \sum^s _{ \alpha < \beta'  } s_ \alpha $.
 In the case when all the $s_ \alpha $'s
 are ordinals,
  the above iterated natural sum has been studied in
 Lipparini \cite{L}.
 It will be probably
 interesting to see which results from \cite{L} extend to the surreal
 framework.
 Notice that, when restricted to ordinals, $\sum^s$ is different from the usual
 transfinite ordinal sum $\sum$. Though the limiting process is the same,
 the successor steps in defining $\sum^s$
 correspond to taking the natural ordinal sum, while in $\sum$
 the usual noncommutative ordinal sum is used.

 Invariance of  $\sum^s $ under permutations
 does not extend beyond $ \omega$; actually, invariance fails already
 at stage $ \omega+1$. Just take $s_0=0$
 and all the other $s_ \alpha $'s to be $1$. Then
 $\sum^s _{ \alpha < \omega +1 } s_ \alpha = \omega +1$,
 but if we permute $s_0$ with $s_ \omega $, we get
 $ \omega$ instead.
 It will be probably interesting to consider transfinite s-sums
  in the case when all the $s_ \alpha $'s are equal.
 Cf. Altman  \cite{A}  and references there for the ordinal case.

 The s-sum equals  the surreal sum in many cases,
 notably, the following corollary is
  immediate from Conway \cite[Chapter 3]{C}
 or Gonshor \cite[Theorem 5.12]{G}.

 \begin{corollary} \labbel{ssnf}
 For every  sequence $(s_ \alpha ) _{ \alpha < \beta } $
 of surreals and  every  sequence
 $(r_ \alpha ) _{ \alpha < \beta } $ of reals, the following identity
 holds
 \begin{equation*}\labbel{sss}
 \sum ^s_{ \alpha < \beta  } \omega ^{s_ \alpha} r_ \alpha
 =
   \sum_{ \alpha < \beta  } \omega ^{s_ \alpha} r_ \alpha,
   \end{equation*}
 provided the latter sum represents the Conway normal form of some surreal.
  \end{corollary}

\subsection*{Comparison with classical series} 
One could try to extend the classical series expansions
 of real analysis to infinite surreal numbers by using the s-sum,
 for example, by considering $f(s) = \sum^s _{n < \omega } \frac{s^n}{n!} $.
  Though $f( \omega )$ gives the expected value $ \omega ^ \omega $, which is equal
 to  $ \exp( \omega )$
 in the sense of the surreal exponentiation,
 on the other hand,   $f( \omega +1) = \omega ^ \omega $,
 too,
 hence series expansions through the s-sum generally
 give unwanted results.
 Actually, as follows from
 the proof of Theorem \ref{thm}, if $e^ r = d$ is dyadic $>1$,
 then $f(r) = d- \nicefrac{1}{ \omega } \not= d$,
 hence series obtained by using the s-sum
 do not always assume  the exact wanted value even for
  real numbers.
 However, there is perhaps  the possibility
 of modifying the s-limit and hence the s-sum
 in order
 to make things work better,
 but this is still to be developed.
 See the last sentence in Remark \ref{rmk}.

 \section{Further remarks and generalizations} \labbel{fur}

 \subsection*{The $s$-limit and the canonical representation of surreals.} \labbel{canrepr}
 According to
 \cite[Theorem 2.8]{G},
 the \emph{canonical representation} of a surreal
 $s$ is   $\{ F \mid G\}$,
 where $F$, $G$, respectively,  are the sets of those surreals
  which are initial segments of $s$ and are
 $<s$, respectively, $>s$.
 If $(s_ \beta ) _{ \beta  < \alpha } $
 is a sequence of surreals and
 $\{ F_ \beta  \mid G_ \beta \} _{ \beta < \alpha } $
 are their respective canonical representations,
 then a representation
 of $\slim _{\beta  < \alpha } s_ \beta  $ is
  $\{ F \mid G\}$,
 where
 $F = \bigcup _{ \beta < \alpha }  \bigcap _{ \beta ' \geq \beta } F _{ \beta '}  $
 and
 $G = \bigcup _{ \beta < \alpha }  \bigcap _{ \beta ' \geq\beta } G _{ \beta '}  $,
 in words, we take as representatives only those elements
 which are eventually in the lower, respectively, upper sets.

 The same remark holds if we start considering, as
 another representation, those
 $F$ and  $G$ which are  the sets of those surreals
 born strictly before $s$ and are
 $<s$, respectively, $>s$.

 However, the above considerations do not always hold
 \emph{for arbitrary} representations of the $s_ \beta $'s.
 For example, for every $n \in \omega $, we have
 $n+1= \{ n \mid \}$,
 but then the above formulas would give
 $\slim _{n < \omega } n = \{ \mid \} = 0 \not = \omega  $.

 \subsection*{Surreal inferior and superior limits} 
  If
 $( s_ \beta ) _{ \beta  < \alpha  } $ is a sequence of
 surreal numbers, we define
 $s= \sliminf_{ \beta  < \alpha  } s_ \beta $
 by defining $s( \gamma )$
 by transfinite induction on $\gamma$,
 simultaneously constructing an auxiliary set $A( \gamma ) $ cofinal in $ \alpha $.
 Suppose that $\gamma$ is an ordinal and that
 both $s( \delta )$ and $A( \delta ) $ have been
 defined, for every $\delta < \gamma $.
 Let $B= \alpha  $, if $\gamma=0$;
 $B= A( \delta ) $, if $\gamma= \delta +1$;
 and $B= \bigcap _{ \delta < \gamma } A( \delta ) $,
 if $\gamma$ is limit.
 We are now ready to define $s(\gamma)$ and $A ( \gamma )$.
 If $B$ is not cofinal in $ \alpha $, we set
 $ s ( \gamma ) $ to be undefined and the construction stops.
 Henceforth, suppose that $B$ is cofinal in $ \alpha $.
 If there is some $  \bar{ \beta } < \alpha  $
 such that
  $s_ \beta ( \gamma ) = +$, for every $ \beta \geq  \bar{ \beta } $,
 $ \beta  \in B$,
 we set  $s( \gamma )= +$ and
 $A( \gamma ) = \{ \beta \in B \mid s_ \beta ( \gamma ) = + \} $.
  If the set of those
 $ \beta  \in B$ such that $s_ \beta ( \gamma ) = -$
 is cofinal in $B$,
 we set $s( \gamma )= -$ and
 $A( \gamma ) = \{ \beta \in B \mid s_ \beta ( \gamma ) = - \} $.
 In all the other cases, we let $s( \gamma )$ be undefined
  and the construction stops.

  In a symmetric way we define
 $\slimsup$.

\subsection*{Extensions and variations on $\slim$}
 As we mentioned in the introduction, Definition \ref{def1} 
can be naturally extended in order to
 deal with  sequences of strings
 of arbitrary symbols, not just  $+$ and  $-$.
 Actually, we shall present a generalization in which
 we take into account
 linear orders, not only
 well-orders.

 \begin{definition} \labbel{def}
  Let $A$ be any set and $L$, $ M$ be linearly ordered
 sets. We shall consider sequences of the form $( a_ \ell) _{\ell \in L} $,
 where each $a_\ell$ is a function from some initial segment
 of $M$ to $A$. Here both the empty set and $M$ itself are considered
 to be initial segments.

 It is useful to visualize the
 $a_\ell$'s as rows in an infinite $L \times M$  matrix
 with possibly empty entries.
 In this sense,
 $a_\ell$ is the $\ell^{\rm th}$ row
 of the matrix, and
 $a_ \ell(m)$ is the element in the  $m^{\rm th}$ column
 of the
  $\ell^{\rm th}$ row.
 (Warning: in the case when $L$ or $M$ is an ordinal,
 the above terminology might be misleading, since, say, $0$ is the
 $1^{\rm st}$ ordinal.)

  We define the \emph{s-limit
   of
  $( a_ \ell) _{\ell \in L} $},
 in symbols, $\slim_{\ell \in L} a_ \ell  $
  as follows.

 If $L$ has a maximum $\bar \ell$, then
 we set
 $\slim_{\ell \in L} a_ \ell  =  a_{\bar \ell}$.

 If $L$ has no maximum, then
 $\slim_{\ell \in L} a_ \ell $ is the function $  a$
 given by the following prescriptions.
 If $m \in M$, we declare $a(m)$ to be defined
 in case there is some $\ell (m) \in L$ such that,
 for every $m' \leq m$ and every $\ell, \ell' \geq \ell(m)$,
 we have
  $ a_{\ell}(m') = a_{\ell\/'}(m')$.
 If this is the case, we let
 $a(m) = a_{\ell(m)}(m) $.
 It is immediate from the definition
 that the domain $\dom (a)$ of $a$ is an initial
 (possibly empty) segment of $M$.
 In particular, the s-limit is \emph{always} and uniquely
 defined.

 Under the above matrix visualization,
 $a(m)$ is defined if there is some $\ell (m) $ such that
 all the columns before (and including) the $m^{\rm th}$ column
 are eventually constant from the $\ell (m) ^{\rm th}$ row on.
 Notice that we could have declared
 $a(m)$ to be defined just in case
 the $m^{\rm th}$ column is eventually
 constant. This would give a different definition of a limit;
 this latter definition has the drawback that
 it does not imply that
 $\dom (a)$ is an initial segment of $M$.
 To avoid the trouble, we can define
 another notion of limits of strings, call it $\slim^\diamond$,
 by declaring $a(m)$ to be defined  if the
  $m^{\rm th}$ column is eventually
 constant \emph{and} all the preceding columns
 are eventually constant, too (the difference with  $\slim$ is that
 here we make no assumption  about  the points from which
 the columns become constant).

 We believe $\slim$ to be more natural than
 $\slim^\diamond$. For sure,
 though the version of Proposition \ref{prop}(a)
  holds for $\slim^\diamond$ with the same proof,
 the analogues of
 Proposition \ref{prop}(b)(c) do not hold for
 $\slim^\diamond$.
 Consider the following increasing
 sequence of strings of length $ \omega+1$:
 $a_0 =-----\dots - $, $a_1=+----\dots + $, $a_3=   ++---\dots - $
 $ a_4 = +++--\dots + $, \dots
 (the main point is that  signs in the last place alternate).
 If $( b _n ) _{ n < \omega } $   is the subsequence consisting of
 the strings with odd index, then
 $\slim _{n  < \omega } a_n = \slim^ \diamond _{n  < \omega } a_n
 = \slim _{n  < \omega } b_n = \omega $,
 but
  $\slim^ \diamond _{n  < \omega } b_n = \omega - 1 \not= \slim^ \diamond _{n  < \omega } a_n$; in particular, $\slim^ \diamond $ and $\slim$ differ on
 $( b _n ) _{ n < \omega } $.
 The counterexample to  \ref{prop}(c) is obtained by considering
 also the  subsequence consisting of
 the strings with even index.
 However, for many  arguments in this note
 the two definitions would turn out  to be essentially equivalent.

 Notice that in the above definitions  we simply discard
 those columns for which the entries are not eventually constant; actually,
 we discard all further columns which follow a column as above.
 In the case when we do not have to discard columns, we speak
  of a full  limit.
 Formally, we say that
 $a$ is the
 \emph{full limit
   of
  $( a_ \ell) _{\ell \in L} $}
 if $\slim_{\ell \in L} a_ \ell  =  a$
 and, in addition,  for every $m \in M$,
 if there is some $\ell (m) \in L$
 such that $a_\ell (m)$ is defined, for
 every $\ell \geq \ell (m) $,
 then   $a(m)$ is defined, too.
 In other words, the s-limit $a$ of a sequence
 is the full limit of the sequences truncated at
 $\dom (a)$, where $\dom (a)$ is the largest possible
 initial segment of $M$ such that the truncated sequences
 do admit a full limit.
 In the above example,
 $ \slim _{n  < \omega } b_n $
 is not a full limit, though, in the sense of
 $ \slim ^ \diamond$,
 $ \slim^ \diamond _{n  < \omega } b_n $
  would actually be a full limit.
 \end{definition}

 The s-limit behaves only partially well with respect to
 string concatenation, that we shall denote by juxtaposition.
 Though $\slim_{\ell \in L} ba_ \ell =b \slim_{\ell \in L} a_ \ell $
 for all strings, it is not necessarily always the case
 that $\slim_{\ell \in L} a_ \ell b = (\slim_{\ell \in L} a_ \ell) b $:
 just take $L= \omega $, $a_n$ of length $n$, for $n \in \omega$,
 and $b$ of length $1$. However,
   $\slim_{\ell \in L} a_ \ell b_ \ell = (\slim_{\ell \in L} a_ \ell)
 \slim_{\ell \in L} b_ \ell $
 holds when all the $a_ \ell $'s have the same length and
 $\slim_{\ell \in L} a_ \ell$ is a full limit.

 If $a= \slim_{\ell \in L} a_ \ell  $
 then either the $a_\ell$'s restricted to   $\dom (a)$ are
 eventually constant, or  $\cf \dom (a) = \cf L$.
 The s-limit of a sequence is equal to the s-limit
 of some subsequence cofinal in $L$, but, in general,
 as we showed before Proposition \ref{prop},
 it is not necessarily
 the case that every cofinal subsequence has the same s-limit.
 However, if $\slim_{\ell \in L} a_ \ell$ is a full limit,
 then every cofinal subsequence has the same limit.

 \begin{remark} \labbel{rmk}
 One can introduce a shorthand for the surreal sign expansion.
 We can consider a surreal number as a sequence of nonzero
 signed ordinals. As in the standard case,
 $0$ is represented by the empty sequence.
 If the first sign in the expansion of the surreal $s$
 is $+$, and we have exactly $\alpha$ consecutive $+$'s at the beginning,
 the first element of the shorthand is $ \alpha $; then if we have a
 certain number $\beta$ of consecutive $-$'s, the second element of the shorthand
 is $-\beta$, and so on. In such a shorthand, ordinals and negated ordinals
 alternate. Let $\sh(s)$ denote
 the shorthand of $s$ in the above sense.
  Taking Definition \ref{def} literally
 would give us strange results,
 such as
 $\slim _{n < \omega } \sh(n) = 0 $.
 However, we can
 adapt the definition
 by
 taking place by place the inferior limit
 for places which consist eventually of positive ordinals,
 taking the superior limit
 for places which consist eventually of negative ordinals,
 and considering a place undefined in the limit if it
 is not of the above kind, with
 the usual convention that  also all the  places which follow
 should be considered undefined.
 Let us denote by $\slim^*$ this modified limit
 acting on shorthands. It gives results different from $\slim$.
 For example, the s-limit of the sequence
 $+-$, $++--$, $+++---$, \dots, is $++++\dots = \omega $,
 while the s-limit$^*$ is $++++\dots ----\dots = \nicefrac{\omega}{2} $.

 Of course, in the general sense of Definition \ref{def},
 $\slim^*$ can be defined when $A$ is a complete  lattice.
 Maybe there are useful variations on the above limit,
 say, using still different representations of surreal numbers.
 This has still to be investigated.
   \end{remark}

\begin{remark} \labbel{uf}    
As is the case for most notions of convergence,
 Definition \ref{def} can be extended to the situation
 when we work modulo some filter.
 Under the notations in
 Definition \ref{def}
 and if $F$ is a filter over $L$, we let
 the \emph{$F$-limit}
 $F $-$\lim a_ \ell$  be the function $a$ such that
 $a(m)$ is defined in case there is $X \in F$ such that,
 for every $m' \leq m$ and every $\ell, \ell' \in X$,
 we have $a_ \ell (m')= a _{\ell'}  (m')$.
 If this is the case, we let $a(m) = a_ \ell (m)$,
 for some $\ell\in X$. Notice that
 the definition of the s-limit in Definition \ref{def}
 is the particular case of the above $F $-limit
 when $F$ is the unbounded filter over $L$.
 The definition in the present remark looks particularly promising, since
 if $A$ is finite, $F$ is an ultrafilter and, for every
 $ m < \omega$, the $a_ \ell$'s are eventually of length $\geq m$,
 then    $F $-$\lim a_ \ell$ has length $ \geq \omega$.
 We could have introduced a variant of the above
 definition (in the same spirit as of $\slim^\diamond$)
   by saying that $a(m)$ is defined
 if, for every $m' \leq m$,
  there is some $X_{m'} \in F$ such that
  $a_ \ell (m')= a _{\ell'}  (m')$,
 for every $\ell, \ell' \in X_{m'}$.
\end{remark} 

\subsection*{Conclusions} 
  In conclusion, at least from some point of view,
  the s-limit appears to be quite unnatural.
 For example, the s-limit of a sum
does not always equal the sum of the s-limits; moreover,
the s-limit of any countable
  sequence with infinitely many positive numbers
 and infinitely many negative numbers, no matter their size,
  is always $0$.
 This is however essentially a consequence of the pleasant fact that the
 s-limit
 is \emph{always} defined.
 On the other hand, the s-limit has a very natural order/string-theoretical
 definition, an interpretation in the game theoretical sense,
 as explained at the beginning of this section,
and is well-behaved with respect to nondecreasing sequences, see 
Proposition \ref{prop}. 
 Moreover,  the s-limit coincides with---or, better, incorporates---
 classical notions of limits in some  significant cases, as shown in
 Theorem \ref{thm} and Section \ref{su},
in particular Corollary \ref{ssnf}.

 \begin{acknowledgement}
Paolo Lipparini thanks Harry Altman
 and Simon Rubinstein-Salzedo for stimulating discussions.
  \end{acknowledgement}

\end{document}